\documentclass[14pt]{amsart}

\usepackage{amsmath}
\usepackage{amssymb}

\usepackage{color}
\usepackage{tikz}
\usepackage{pgfplots}

\newcommand{\QQ}{{\mathbb Q}}
\newcommand{\ZZ}{{\mathbb Z}}

\newcommand{\NN}{\mathbb N}

\newtheorem{theorem}{Theorem}[section]
\newtheorem{lem}[theorem]{Lemma}

\newtheorem{prop}[theorem]{Proposition}

\newtheorem{definition}[theorem]{Definition}
\newtheorem{example}[theorem]{Example}

\allowdisplaybreaks

\begin{document}

\title{Pythagorean Triples in the Fibonacci Model Set}

\author{Sarah Marklund}
\address{Department of Mathematics and Statistics,
         MacEwan University, \newline
\hspace*{\parindent}10700  104 Avenue,
         Edmonton, AB, Canada, T5J 4S2}
\email{marklunds2@mymacewan.ca}

\author{Evangeline Tweddle}
\address{Department of Mathematics and Statistics,
         MacEwan University, \newline
\hspace*{\parindent}10700  104 Avenue,
         Edmonton, AB, Canada, T5J 4S2}
\email{tweddlee2@mymacewan.ca}

\begin{abstract}
In this paper we give a description of all Pythagorean triples in the ring ${\ZZ}[\tau]$. We also consider triples in the Fibonacci model set which satisfy the Diophantine equations arising from Fermat's Last Theorem. Examples are provided, including a counterexample to Fermat's Last Theorem for the third degree in the Fibonacci model set. 
\end{abstract}



\maketitle

\section{Introduction}
Quasicrystals were first discovered in 1982 \cite{Shecht}. Since then, model sets have been utilized as a way to represent and study them, as well as having applications in other studies of aperiodicity \cite{Moody}. The Fibonacci model set is one of the more easily understood and deeply studied model sets \cite{Fib2}, \cite {Fib1}.

In this paper we study solutions to Pythagoras' equation,   $a^2 + b^2 = c^2$, in ${\ZZ}[\tau]$ and the Fibonacci model set. Quite similarly, counterexamples to Fermat’s Last Theorem, proven by Andrew Wiles \cite{FLT}, which states that the equation $a^n + b^n = c^n$ has no integer solutions when $n>2$, will also be discussed. 

The order of the paper is as follows. Section 2 discusses preliminaries. In Section 3 we give a description of all Pythagorean triples in ${\ZZ}[\tau]$. Section 4 gives a description of all triples that are in the Fibonacci model set that satisfy the Diophantine equations arising from Fermat's Last Theorem. We end with examples of these triples and show that there are solutions to the third degree Diophantine equation.

\section{Preliminaries}

To define the Fibonacci model set, we utilize the Fibonacci substitution \cite[Example 4.6]{TAO1}: 
\begin{align*}
    a \rightarrow ab \\
    b \rightarrow a
\end{align*}

If we start with $a|a$, with the bar representing the origin,  and keep applying the substitution rule on each side of the origin, we obtain 
\vspace{.1in}

$a|a \rightarrow ab|ab \rightarrow aba|aba \rightarrow abaab|abaab \rightarrow abaababa|abaababa \rightarrow ...$

\vspace{.1in}

Notice that every element is contained in the second substitution forward. Using the corresponding substitution matrix 
$\begin{pmatrix}
1 & 1 \\
1 & 0 \end{pmatrix}$, 
which has an eigenvalue of $\tau$, we can convert this symbolic substitution into a subset of the real line by viewing $a$ as an interval of length $\tau$, $b$ as an interval of length $1$, and taking endpoints. 

\vspace{.1in}
   
If we do this on every second sequence, these points make up the Fibonacci model set:
\vspace{.1in}

$a|a \rightarrow aba|aba \rightarrow abaababa|abaababa \rightarrow ...$

\vspace{.1in}

becomes
\begin{align*}
\{-\tau, 0, \tau\} &\rightarrow \{-2\tau-1, -\tau-1, -\tau, 0, \tau, \tau+1, 2\tau+1\} \\
&\rightarrow \{-5\tau-3, -4\tau-3, -4\tau-2, -3\tau-2, -2\tau-2, -2\tau-1,-\tau-1,-\tau, \\
& \quad \quad 0, \tau, \tau+1, 2\tau+1, 3\tau+1, 3\tau+2, 4\tau+2, 4\tau+3, 5\tau+3\} \\
&\rightarrow ...
\end{align*}

\vspace{.1in}

We need a little more theory before introducing a more workable description of the Fibonacci model set. 

\subsection{Basic Facts About The Ring ${\ZZ}[\tau]$}
\hfill\\
\vspace{0.01cm}

Here we list basic facts about ${\ZZ}[\tau]$. For further details see \cite{Hunger}. 

$\tau=\frac{1+\sqrt{5}}{2}$ and  $\tau'=\frac{1-\sqrt{5}}{2}$ are the roots of $x^2-x-1=0$, the eigenvalues of the Fibonacci substitution matrix. They satisfy the following relations:
\begin{align*}
\begin{cases}
   \tau^2=\tau+1 \\  
\tau+\tau'=1 \\
\tau\tau'=-1 \\
\tau'^2=\tau'+1.
\end{cases}
\end{align*}
    
Denote by 
$${\ZZ}[\tau] = \{m+n\tau \, | \, m,n \in {\ZZ}\},$$
the ring ${\ZZ}$ adjoin $\tau$. From the above properties of $\tau$ it is very easy to see that for $x,y \in {{\mathbb Z}} [\tau]$, then $x+y, x-y, xy \in {{\mathbb Z}} [\tau]$.

\vspace{0.1in}

Next, let us consider the \emph{Galois conjugation} $\sigma: {\ZZ}[\tau] \mapsto {\ZZ}[\tau]$ such that,
$$\sigma(m+n\tau)= m+n\tau'.$$
Note here that 
\begin{align*} 
&\tau+\tau'=1 \implies \tau'=1-\tau \\
\implies &\sigma (m+n\tau)= (m+n)-n\tau \in {{\mathbb Z}} [\tau].
\end{align*}
Thus, $\sigma$ is well defined.
In particular, the function $\sigma$ is a ring homomorphism and an involution. Therefore for $x, y \in {{\mathbb Z}} [\tau]$ we have
\begin{align*}
&\sigma (x+y)= \sigma(x) + \sigma(y) \\
&\sigma (xy)= \sigma(x) \sigma(y) \quad \textrm{and} \\
&\sigma(\sigma(x))=x.
\end{align*}

As mentioned, the Fibonacci model set can be described in a much simpler manner than the iterative process given above. This will prove very useful for the purposes of this paper. 

\begin{definition}[Fibonacci Model Set] \cite[Example 7.3]{TAO1} \label{def:2.3}
Another description of the Fibonacci Model Set is given by ${\displaystyle \Lambda} = \big\{x \in {\ZZ}[\tau] \,|\, \sigma(x) \in [-1, \tau-1)\big\}$. This description of the Fibonacci Model Set is referred to as a model set.
\end{definition}

\begin{lem} \label{lem:2.4}
If $x \in \Lambda$ then $x\tau \in {\displaystyle \Lambda}$. 
\end{lem}

\begin{proof}
By Definition \ref{def:2.3}, $\sigma(x) \in [-1, \tau-1)$.

\vspace{.1in}

Then $\sigma(x\tau) = \sigma(x)\tau' \in (-\tau'^2, -\tau'] \subseteq [-1, \tau-1).$ Therefore, $x\tau \in \Lambda$. 

\end{proof}
Define the norm of ${\ZZ}[\tau]$, $N:{\ZZ}[\tau] \rightarrow {\ZZ}$, by $N(x)=x\sigma(x)$.\\
Thus, \begin{align*} N(m+n\tau)&=(m+n\tau)\sigma(m+n\tau)\\
&=m^2+mn\tau'+mn\tau+n^2\tau\tau' \\
&=m^2+mn\tau'+mn\tau-n^2 \\
&=m^2+mn-n^2 \in {\ZZ}.
\end{align*}
Since $\sigma$ is multiplicative, so is $N$.

It should also be noted that ${\ZZ}[\tau]$ is a Unique Factorization Domain (UFD). The following well-known theorem proves this fact. 

\begin{theorem} \label{thm:2.3}
${\ZZ}[\tau]$ is a UFD.
\end{theorem}

\begin{proof}
Let $x, y \in {{\mathbb Z}}[\tau], y \neq 0$. Then,
$$ \frac{x}{y} = \frac{x\sigma(y)}{N(y)} = \frac{m+n\tau}{N(y)} $$
for some $m, n, l = N(y) \in {\ZZ}$.
Therefore $\exists \, a, b \in {\QQ}$ such that 
$$\frac{x}{y} = a+b\tau.$$
Pick integers $k, l$ such that, 
$$ \left|a-k\right| \leq \frac{1}{2} $$
$$ \left|b-l\right| \leq \frac{1}{2}. $$
Set $q=k+l\tau$. Then,
$$ \frac{x}{y} - q = (a-k) + (b-l)\tau $$ 
and 
$$ - \frac{1}{2} < (a-k)^2 + (a-k)(b-l) - (b-l)^2 < \frac{1}{2}. $$
This gives
\begin{align*}
|N(x-qy)| &< \frac{1}{2} |N(y)| \\
&< |N(y)|.
\end{align*}
Therefore ${\ZZ}[\tau]$ with $\sigma(y) = \left|N(y)\right|$ is a Euclidean Domain, and hence a UFD.

\end{proof}

Another important fact required for this paper is the irreducibility of $2$ in ${\ZZ}[\tau]$. 

\begin{prop} \label{fact:2.4}
2 is irreducible in ${{\mathbb Z}}[\tau]$.
\end{prop}

\begin{proof}
Assume by contradiction that 2 is not irreducible. Then, $\exists \, x, y \in {{\mathbb Z}}[\tau]$ such that $N(x), N(y) \neq 0, \pm1$ and $2=xy$. Then, $4=N(2)=N(xy)=N(x)N(y)$. Since $N(x)\neq\pm1, N(y)\neq \pm1$, we must have $N(x)=\pm 2$. Let $x=m+n\tau$. Then, \begin{align*}
\pm 2&= N(x)=m^2+mn-n^2 \\
\implies \pm8&= 4m^2+4mn-4n^2 \\
\implies \pm8&= 4m^2+4mn+n^2-5n^2 \\
\implies \pm8& =(2m+n)^2-5n^2 \\
\implies \pm3 &=(2m+n)^2 \mbox{ (mod 5)}.
\end{align*}

But it is easy to see that the equation $z^2 = \pm 3 \pmod{5}$ has no solution, which is a contradiction.

\end{proof}

\section{Description of Triples in ${{\ZZ}} [\tau]$}

In any commutative ring it is straightforward to produce many examples of Pythagorean triples. The following result is folklore.

\begin{prop}
For $l, m, n \in {\ZZ}[\tau]$, define $x, y, z$ in the following manner:
\begin{align*}
    \begin{cases}
        x=\pm2 lmn \\
        y= l(m^2-n^2) \\
        z= l(m^2+n^2).
    \end{cases}
\end{align*}
Then, $x^2+y^2=z^2$. 
\end{prop}

\begin{proof}
We can show this equation holds for the given definition of $x, y, z$:

\begin{align*}
    x^2+y^2 &= (\pm 2lmn)^2+(l(m^2-n^2))^2 \\
    &= 4l^2m^2n^2+l^2(m^4-2m^2n^2+n^4)=l^2(m^4+2m^2n^2+n^4) \\
    &= (l(m^2+n^2))^2 = z^2.
\end{align*}

\end{proof} 
Note that multiple choices of $l,m$ and $n$ can give the same triple. What is harder to show is that this description, including when the roles of $x$ and $y$ are interchanged, gives all Pythagorean triples in a UFD extension of ${\ZZ}$ where $2$ remains prime.

We follow here a standard folklore proof, ignoring uniqueness and adapting it to the ${\ZZ}[\tau]$ context. First, we need a technical lemma. 
\begin{lem}
Let $a,b \in {{\ZZ}} [\tau]$ be non-zero such that
\begin{align*}
    \begin{cases}
    \mbox{gcd} (a,b)=1 \\
    a \cdot b=c^2
    \end{cases}
\end{align*}
Then, there exists unit $u$ in ${\ZZ}[\tau]$ and $m,n \in {\ZZ}[\tau]$ such that 
\begin{align*}
    \begin{cases}
    a= u \cdot m^2 \\
    b= u \cdot n^2
    \end{cases}
\end{align*}
\end{lem}
\begin{proof}
Let $P$ be the set of all prime elements in ${\ZZ}[\tau]$. Define an equivalence relation on $P$ via $p \sim q$ if $q$ and $p$ are associated. Fix some system of representatives $Q$ for this equivalence relation, meaning $Q$ consists of exactly one element from each equivalence class. 
Then, since ${\ZZ}[\tau]$ is a UFD as proved in Theorem \ref{thm:2.3}, there exists unique $p_1, ... , p_k$ and $q_1, ... , q_j \in Q$ and units $u,v$ such that 
\begin{align*}
    a= u \cdot p_1^{\alpha_1} \cdot \ldots \cdot p_k^{\alpha_k} \\
    b= v \cdot q_1^{\beta_1} \cdot \ldots \cdot q_j^{\beta_j}
\end{align*}
Since gcd $(a,b) =1$, then $p_1, ... , p_k, q_1, ... , q_k$ are pairwise distinct, i.e. $\forall \, i \neq j$ we have $p_i \neq q_j$.
\vspace{.1in}

Then, $c^2= u \cdot v \cdot p_1^{\alpha_1}\cdot \ldots \cdot p_k^{\alpha_k} \cdot q_1^{\beta_1}\cdot \ldots \cdot q_j^{\beta_j}$ 
\vspace{.1in}

 We show now that each power is even. Indeed, by the Fundamental Theorem of Arithmetic, 
$c$ will have a unique prime decomposition $c= w \cdot r_1^{m_1} \cdot \ldots \cdot r_s^{m_s}$ for prime elements $r_1, ... , r_s \in Q$ and unit $w$. \\
Then, $c^2=w^2 \cdot r_1^{2m_1} \cdot \ldots \cdot r_s^{2m_s}$ and $c^2= u \cdot v \cdot p_1^{\alpha_1} \cdot \ldots \cdot q_j^{\beta_j}$. \\
The uniqueness of prime factorization and the choice of $Q$ gives $uv=w^2$ and each $p_l^{\alpha_l}$ or $q_l^{\beta_l}$ must be one of $r_t^{2m_t}$, meaning the powers are even. 
\vspace{.1in}

We can also note that 
$$
    u \cdot v = w^2 \implies v = w^2 \cdot u^{-1} = u (w \cdot u^{-1})^2.
$$
Then, $v=uk^2$ where $k=w \cdot u^{-1}$.
\vspace{.1in}

So, since $\alpha_1, ... , \alpha_k, \beta_1,...,\beta_k$ are even, 
\begin{align*}
    \alpha_i=2\alpha_i' \\
    \beta_i=2\beta_i'
\end{align*}
where $\alpha_i', \beta_i' \in {\NN}$. Therefore,  
\begin{align*}
    a= u \cdot p_1^{\alpha_1} \cdot \ldots \cdot p_k^{\alpha_k} = u \cdot \left(\underbrace{p_1^{\alpha_1'} \cdot \ldots \cdot p_k^{\alpha_k'}}_{\text{m}}\right)^2
    \end{align*}
    \begin{align*}
    b= v \cdot q_1^{\beta_1} \cdot \ldots \cdot q_j^{\beta_j}&= v \cdot (q_1^{\beta_1'} \cdot \ldots \cdot q_k^{\beta_j'})^2 \\
    &=u \cdot k^2 \cdot (q_1^{\beta_1'} \cdot \ldots \cdot q_k^{\beta_j'})^2\\
    &= u \cdot \left(\underbrace{k \cdot q_1^{\beta_1'} \cdot \ldots \cdot q_k^{\beta_j'}}_{\text{n}}\right)^2
\end{align*}
\end{proof}
With this we can now prove the main result of this section.
\begin{theorem}
Every Pythagorean Triple in ${\ZZ}[\tau]$ is of the form 
\begin{align*}
    \begin{cases}
        x=\pm2 lmn \\
        y= l(m^2-n^2) \\
        z= l(m^2+n^2)
    \end{cases}
\end{align*}
or where $x$ and $y$ are interchanged.
\end{theorem}

\begin{proof}
Let $x,y,z \in {{\ZZ}} [\tau]$ be such that $$x^2+y^2=z^2.$$
Then, 
$$x^2=z^2-y^2=(z-y)(z+y).$$

Let $d=\mbox{gcd}(z-y,z+y)$. Then, $\exists \, a,b \in {{\ZZ}} [\tau]$ such that 
\begin{align*}
    \begin{cases}
    z-y=da \\
    z+y=db \\
    \mbox{gcd}(a,b)=1.
    \end{cases}
\end{align*}
Therefore, 
\begin{align*}
x^2= dadb= d^2 \cdot ab\\
\implies d^2|x^2 \implies d|x.
\end{align*}
So $x=dc$ for some $ c \in {{\ZZ}}[\tau]$. Then,
\begin{align*}
d^2c^2=x^2=d^2\cdot ab \\
\implies 
\begin{cases}
\mbox{gcd} (a,b)=1 \\
    a \cdot b=c^2.
\end{cases}
\end{align*}
The previous lemma gives that there exists unit $u \in {\ZZ}[\tau]$ and $m,n \in {\ZZ}[\tau]$ such that 
\begin{align*}
    \begin{cases}
    a=u \cdot m^2 \\
    b= u \cdot n^2
    \end{cases}
\end{align*}
and hence 
\begin{align*}
    \begin{cases}
    z+y= da= du \cdot m^2 \\
    z-y= db= du \cdot n^2
    \end{cases}
\end{align*}
which implies
\begin{align}
    \begin{cases}
    2z= du \cdot (m^2+n^2) \\
    2y= du \cdot (m^2-n^2)
    \end{cases}
\end{align}
Since 2 is a prime element, we have two cases: \\
\underline{Case 1}: $2 \mid du$ 
\vspace{.1in}

Let $du=2l$. Then, Equation 1 becomes 
\begin{align*}
    \begin{cases}
    2z &= 2l \cdot (m^2+n^2) \\ 
    2y &= 2l \cdot (m^2-n^2) \end{cases}
    \quad \rightarrow \quad \begin{cases} 
    z &= l \cdot (m^2+n^2) \\ 
    y &= l \cdot (m^2-n^2) \end{cases} 
\end{align*}
Then, \begin{align*}
    x^2= (z+y)(z-y) &= dum^2 \cdot dun^2 =2lm^2 \cdot 2ln^2  =(2lmn)^2
    \end{align*}
\begin{align*}
    \begin{cases}
    x= \pm 2lmn \\
    y= l (m^2-n^2) \\
    z= l (m^2+n^2)
    \end{cases}
\end{align*}

\underline{Case 2}: $2 \nmid du$
\vspace{.1in}

Since $2y=du (m^2-n^2)$ and 2 is prime, then

\begin{equation*}
  2 \mid m^2 - n^2 = (m-n)(m+n) 
\end{equation*}
and so 
$2 \mid m-n$ or $2 \mid m+n$. 
\vspace{.1in}

By replacing $n$ with $ -n$, we need discuss only the case $2|m-n$. 
\vspace{.1in}

Let $m-n=2m'$. \\
Then, $m+n= m-n+2n = 2m'+2n = 2(m'+n)$.
\vspace{.1in}

Set $n'=m'+n$ and so
$m+n=2n'$. 
\vspace{.1in}

Set $l'=du$. Then, Equation 1 becomes 
\begin{align*}
    \begin{cases}
    2z= l'(m^2+n^2) \\
    2y= l'(m^2-n^2)
    \end{cases}
\end{align*}
Now, $m^2-n^2=(m-n)(m+n)=2m'2n'$.
\vspace{.1in}

Next, \begin{align*}
m^2+n^2&= \frac{1}{2}\big((m-n)^2+(m+n)^2\big) \\
&= \frac{1}{2}(4m'^2+4n'^2) \\
&= 2(m'^2+n'^2) \\
\end{align*}
Hence,
\begin{align*}
\begin{cases}
z= l' (m'^2+n'^2) \\
y= 2l' m' n'
\end{cases}
\end{align*}
Finally,
\begin{align*}
    x^2=(z-y)(z+y) &= l'^2m^2n^2 =l'^2 (mn)^2 \\
    = l'^2(\frac{1}{4}((\underbrace{m+n}_{\text{2n'}})^2-(\underbrace{m-n}_{\text{2m'}})^2))^2 & =l'^2 (\frac{1}{4}(4n'^2-4m'^2))^2 =l'^2 (n'^2-m'^2)^2\\
    &\implies
    \begin{cases}
    x= \pm l' (n'^2-m'^2) \\
    y= 2 l'm'n' \\
    z= l' (m'^2+n'^2)
    \end{cases}
\end{align*}
\end{proof}

\section {Pythagorean Triples in the Fibonacci Model Set}

We now turn to triples in the Fibonacci model set which satisfy the Diophantine equations arising in Fermat's Last Theorem. We will end with examples of triples in the Fibonacci model set satisfying Pythagoras' equation and the third degree Diophantine equation. 
\vspace{.1in}

The following lemma is straightforward and its proof will not be given.

\begin{lem} \label{lem:4.2}
If $z\in {{\mathbb Z}}[\tau] \backslash \{0\}$, then there exists an $n_z \in {\ZZ}$ such that $n \geq n_z$ if and only if $|z(\tau')^n|<\tau-1$.
\end{lem}

We can now show that every non-trivial solution to Fermat's Last Theorem produces solutions inside the so-called window of the Fibonacci model set.

\begin{lem} \label{lem:4.3}
If $x, y, z \in {\ZZ}[\tau] \backslash \{0\}$ satisfy $x^k + y^k = z^k$ for some $k \geq 2$, then there exists an $n_0$ such that
\begin{align*}
    \begin{cases}
        x_1 = x \cdot \tau'^n \\
        y_1 = y \cdot \tau'^n \\
        z_1 = z \cdot \tau'^n
    \end{cases}
    \in [ -1, \tau - 1)
\end{align*}
and $x_1^k + y_1^k = z_1^k$ for all $n \geq n_0$.
\end{lem}

\begin{proof}
Assume $x, y, z \in {\ZZ}[\tau] \backslash \{0\}$ satisfy $x^k + y^k = z^k$ for some $k \geq 2$. 
 
\vspace{.1in}

Then, by Lemma \ref{lem:4.2}, since $z \in {\ZZ}[\tau] \backslash \{0\}$, there exists an $n_z \in {\ZZ}$ such that for all $n \geq n_z$, $|z \cdot (\tau')^n| < \tau-1$. 
\vspace{.1in}

By the same reasoning, since $x,y \in {\ZZ}[\tau] \backslash \{0\}$, there exists $n_x,n_y \in {\ZZ}$ such that $|x \cdot (\tau')^n|<\tau-1$ for all $n \geq n_x$ and $|y \cdot (\tau')^n|<\tau-1$ for all $n \geq n_y$. 
\vspace{.1in}

We then denote the largest of $n_x,n_y,n_z$ as $n_0$. \\
Hence, $|x \cdot (\tau')^n|,|y \cdot (\tau')^n|, |z \cdot (\tau')^n| < \tau-1$ for all $n \geq n_0$. 

\vspace{.1in}
Therefore, for all $n \geq n_0$, 
\begin{align*}
    \begin{cases}
        x_1 = x \cdot \tau'^n \\
        y_1 = y \cdot \tau'^n \\
        z_1 = z \cdot \tau'^n
    \end{cases}
    \in (1-\tau, \tau - 1) \subseteq [-1, \tau-1).
\end{align*}

\vspace{.1in}
Lastly, since $x^k + y^k = z^k$ 
\begin{align*}
&\implies \tau'^{kn} (x^k + y^k) = \tau'^{kn} (z^k) \\
&\implies x^k\tau'^{kn} + y^k\tau'^{kn} = z^k\tau'^{kn} \\
&\implies (x \cdot \tau'^n)^k + (y \cdot \tau'^n)^k = (z \cdot \tau'^n)^k \\
&\implies x_1^k + y_1^k = z_1^k.
\end{align*}

\end{proof}

We can now show that if there is a ${\ZZ}[\tau]$ solution to the Diophantine equation arising in Fermat's Last Theorem then there is a Fibonacci model set solution. 

\begin{theorem} \label{theorem:4.4}
If $a, b, c \in {{\mathbb Z}}[\tau] \backslash \{0\}$ satisfy $a^k+b^k=c^k$ for some $k \geq 2$, then there exists an $n_0$ such that $a \cdot \tau^n, b \cdot \tau^n, c \cdot \tau^n \in \Lambda$, the Fibonacci model set, and $(a \cdot \tau^n)^k+(b \cdot \tau^n)^k=(c \cdot \tau^n)^k$ for all $n \geq n_0$.
\end{theorem}

\begin{proof}
Assume $a, b, c \in {\ZZ}[\tau] \backslash \{0\}$ satisfy $a^k + b^k = c^k$ for some $k \geq 2$. Define $x = \sigma(a), y = \sigma(b), z = \sigma(c)$. 

Thus, $x,y,z \in {\ZZ}[\tau] \backslash \{0\}$ and
\begin{align*}
   x^k + y^k &= \sigma(a)^k + \sigma(b)^k = \sigma(a^k + b^k) \\
    &= \sigma(c^k) = \sigma(c)^k =z^k.
\end{align*}

By Lemma \ref{lem:4.3}, there exists an $n_0 \in {\ZZ}$ such that
\begin{align*}
    \begin{cases}
    x_1=x \cdot \tau'^n \\
    y_1=y \cdot \tau'^n \\
    z_1=z \cdot \tau'^n \\
    \end{cases}
    \in [-1, \tau-1)
\end{align*}
for all $n \geq n_0$. Hence, for every $n \geq n_0$ we have $\sigma(a\tau^n) = x\tau'^n$, $\sigma(b\tau^n) = y\tau'^n$, $\sigma(c\tau^n) = z\tau'^n$. Definition \ref{def:2.3} then gives that $a\tau^n, b\tau^n, c\tau^n$ are in the Fibonacci model set. 

Lastly, $(a\tau^n)^k + (b\tau^n)^k = (a^k + b^k)\tau^{nk} = c^k\tau^{nk} = (c\tau^n)^k$.

\end{proof}

The other direction can also be proven. 

\begin{lem} \label{lem:4.6}
If $x,y,z \in {{\ZZ}}[\tau] \backslash \{0\}$ satisfy $x^k+y^k=z^k$ for some $k \geq 2$, then there exists an $n_1$ such that
\begin{align*}
    \begin{cases}
    x_1= x \cdot \tau'^n \\
    y_1=y \cdot \tau'^n \\
    z_1=z \cdot \tau'^n
    \end{cases}
    \not\in [-1,\tau-1)
\end{align*}
and $x_1^k+y_1^k=z_1^k$
for all $n<n_1$.
\end{lem}
\begin{proof}
Assume $x,y,z \in {{\ZZ}}[\tau] \backslash \{0\}$ satisfy $x^k+y^k=z^k$ for some $k \geq 2$. 
\vspace{.1in}

By the same reasoning as in Lemma \ref{lem:4.2}, there must exist a point $a$ such that $|x||\tau'|^a=1$. Then, since $|\tau'|<1$, by the properties of exponential functions, $|x||\tau'|^n$ is monotone decreasing. Hence, $|x||\tau'|^n>1$ for all $n<a$.
\vspace{.1in}

Next, by the same reasoning, there exist points $b$ and $c$ such that $|y||\tau'|^b=1$ and hence $|y||\tau'|^n>1$ for all $n<b$ and $|z||\tau'|^c=1$ and hence $|z||\tau'|^n>1$ for all $n<c$.
\vspace{.1in}

Now let $n_1$ be an integer such that $n_1 < \text{min}\{a, b,c\}$. We can also define $n_1 = \lfloor \text{min}\{a, b,c \} \rfloor -1$. This implies $|x||\tau'|^n,|y||\tau'|^n, |z||\tau'|^n>1$ for all $n<n_1$. Therefore, there exists an $n_1 \in {\ZZ}$ such that 
\begin{align*}
    \begin{cases}
    x_1= x \cdot \tau'^n \\
    y_1=y \cdot \tau'^n \\
    z_1=z \cdot \tau'^n
    \end{cases}
    \not\in [-1,\tau-1)
\end{align*}
for all $n<n_1$.
\vspace{.1 in}

Lastly, $x^k+y^k=z^k$
\begin{align*}
    &\implies \tau'^{kn}(x^k+y^k)=\tau'^{kn}(z^k) \\
    &\implies (x \cdot \tau'^n)^k+(y \cdot \tau'^n)^k=(z \cdot\tau'^n)^k \\
    &\implies x_1^k+y_1^k=z_1^k
\end{align*}
\end{proof}

We use this to show that every solution to such Diophantine equations in ${\ZZ}[\tau]$ lead to a solution entirely outside the Fibonacci model set. 

\begin{theorem} \label{theorem:4.7}
If $a,b,c \in {{\ZZ}}[\tau] \backslash \{0\}$ satisfy $a^k+b^k=c^k$ for some $k \geq 2$, then there exists an $n_1 \in {\ZZ}$ such that $a \cdot \tau^n, b \cdot \tau^n, c \cdot \tau^n \not\in \Lambda$, the Fibonacci model set, and $(a \cdot \tau^n)^k+(b \cdot \tau^n)^k=(c \cdot \tau^n)^k$ for all $n<n_1$.
\end{theorem}

\begin{proof}
Assume $a,b,c \in {{\ZZ}}[\tau] \backslash \{0\}$ satisfy $a^k+b^k=c^k$ for some $k \geq 2$. Define $x=\sigma(a), y=\sigma(b), z=\sigma(c)$.

Thus, $x,y,z \in {\ZZ}[\tau] \backslash \{0\}$ and 
\begin{align*}
   x^k + y^k &= \sigma(a)^k + \sigma(b)^k = \sigma(a^k + b^k) \\
    &= \sigma(c^k) = \sigma(c)^k = z^k.
\end{align*}

By Lemma \ref{lem:4.6}, we know there exists an $n_1 \in {\ZZ}$ such that 
\begin{align*}
    \begin{cases}
    x_1=x \cdot \tau'^n \\
    y_1=y \cdot \tau'^n \\
    z_1=z \cdot \tau'^n \\
    \end{cases}
    \not\in [-1, \tau-1)
\end{align*}
for all $n<n_1$. Hence, for every $n < n_1$ we have $\sigma(a\tau^n) = x\tau'^n$, $\sigma(b\tau^n) = y\tau'^n$, $\sigma(c\tau^n) = z\tau'^n$. Definition \ref{def:2.3} gives that $a \cdot \tau^n, b \cdot \tau^n, c \cdot \tau^n$ are not in the Fibonacci model set. 

Lastly, $(a\tau^n)^k + (b\tau^n)^k = (a^k + b^k)\tau^{nk} = c^k\tau^{nk} = (c\tau^n)^k$.

\end{proof}

We now conclude that every solution to such a Diophantine equation in ${\ZZ}[\tau]$ sits in an infinite set of solutions, given by multiplying by $\tau^n$, $n \in {\ZZ}$. This set is divided into two parts, the right side is in the Fibonacci model set, and the left side is not. 

\begin{theorem} \label{prop:4.8}
Suppose $a,b,c \in {\ZZ} [\tau] \backslash \{0\}$ satisfy $a^k+b^k=c^k$. Then, there exists $N \in {\ZZ}$ such that $a \cdot \tau^n, b \cdot \tau^n, c \cdot \tau^n \in \Lambda$, the Fibonacci model set, if and only if $n \geq N$.
\end{theorem}

\begin{proof}
Assume $a,b,c \in {\ZZ}[\tau] \backslash \{0\}$ satisfy $a^k+b^k=c^k$.
\vspace{.1in}

We then know from Theorem \ref{theorem:4.4} that there exists an $n_0 \in {\ZZ}$ such that $a \cdot \tau^n, b \cdot \tau^n, c \cdot \tau^n \in \Lambda$  for all $n \geq n_0$. We also know from Theorem \ref{theorem:4.7} that there exists an $n_1 \in {\ZZ}$ such that $a \cdot \tau^n, b \cdot \tau^n, c \cdot \tau^n \not\in \Lambda$ for all $n<n_1$.
\vspace{.1in}

Next, consider the finite set A=\{$n_1, ... , n_0$\}, and notice that A is nonempty since $n_1 \leq n_0$. There must exist a smallest element $N \in A$ such that  $a \cdot \tau^N, b \cdot \tau^N, c \cdot \tau^N \in \Lambda$.
\vspace{.1in}

Thus, $a \cdot \tau^n, b \cdot \tau^n, c \cdot \tau^n \in \Lambda$ for all $n \geq N$ as multiplying any element in the Fibonacci model set by $\tau$ will give an element that is still in the Fibonacci model set, see Lemma \ref{lem:2.4}.

\vspace{.1in}

Since $N$ is the smallest element in $A$ for which $a \cdot \tau^n, b \cdot \tau^n, c \cdot \tau^n \in \Lambda$ , then by definition, $a \cdot \tau^n, b \cdot \tau^n, \mbox{or }  c \cdot \tau^n \not\in \Lambda$ for all $n_1 \leq n < N$. Therefore, $a \cdot \tau^n, b \cdot \tau^n, c \cdot \tau^n \in \Lambda$ if and only if $n \geq N$.

\end{proof} 

As a consequence we get the following theorem. 
\vspace{.1in}

\begin{theorem} 
~\begin{itemize} 
\item[(a)] $(x,y,z)$ is a Pythagorean triple in the Fibonacci model set if and only if there exists some $l,m,n \in {\ZZ}[\tau]$ such that
\begin{align*}
    \begin{cases}
        x=\pm2 lmn \\
        y= l(m^2-n^2) \\
        z= l(m^2+n^2)
    \end{cases}
\end{align*}
and $\sigma(x), \sigma(y), \sigma(z) \in [-1, \tau-1)$.

\item[(b)] For each $l,m,n \in {\ZZ}[\tau]$ there exists some $N \in {\ZZ}$ such that
\begin{align*}
\begin{cases}
x &= \pm 2 (l \tau^k)mn \\
y &= (l \tau^k)(m^2-n^2) \\
z&= (l \tau^k)(m^2+n^2)
\end{cases}
\end{align*}
is a Pythagorean triple in the Fibonacci model set for all $k > N$.
\end{itemize}
\end{theorem}

\subsection{Examples in The Fibonacci Model Set: Pythagoras' Theorem}
\hfill\\
\vspace{0.01cm}

Particular solutions to Pythagoras' equation in ${\ZZ}[\tau]$ and the third degree Diophantine equation in Fermat's Last Theorem are shown below. 

\begin{example}
Let $x = \tau$, $y = 2\tau$, and $z = 2+\tau$.
Then
\begin{align*}
    (\tau)^2 + (2\tau)^2 &= 5\tau^2 = 5(\tau+1) \\
    &= 4 + 4\tau + \tau + 1 = 4 + 4\tau + \tau^2 = (2+\tau)^2.
\end{align*}

In order to give an example in the Fibonacci model set where $\sigma (\cdot)$ falls into the window $[-1, \tau - 1)$ we will multiply $\sigma(x) = \tau', \sigma(y) = 2\tau', \sigma(z) = 2+\tau'$ by $(\tau')^n = \big(\frac{1 - \sqrt{5}}{2}\big)^n$ for some $n \in {\ZZ}$. After testing, we have found that for $n = 2$ the results fall into the window, as shown below.

\begin{align*}
    (\tau')\cdot(\tau')^2 &\simeq -0.236 \\
    (2\tau')\cdot(\tau')^2 &\simeq -0.472 \\
    (2+\tau')\cdot(\tau')^2 &\simeq 0.528.
\end{align*}

Then we have

\begin{align*}
     \big((\tau')\cdot(\tau')^2\big)^2 + \big((2\tau')\cdot(\tau')^2\big)^2 &= \big((2+\tau')\cdot(\tau')^2\big)^2 \\
     \implies (1+2\tau')^2 + (2+4\tau')^2 &= (3+4\tau')^2
\end{align*}

and $(1+2\tau, 2+4\tau, 3+4\tau)$ is an example of a triple in the Fibonacci model set satisfying Pythagoras' equation. Moreover, $\sigma(2+\tau) = 2 + \tau' \notin [-1, \tau-1)$ and so by Theorem \ref{prop:4.8} $((\tau)\tau^n)^2 + ((2\tau)\tau^n)^2 = ((2+\tau)\tau^n)^2$ is in the Fibonacci model set if and only if $n \geq 2$.

This triple is small enough to graph on the Fibonacci model set.

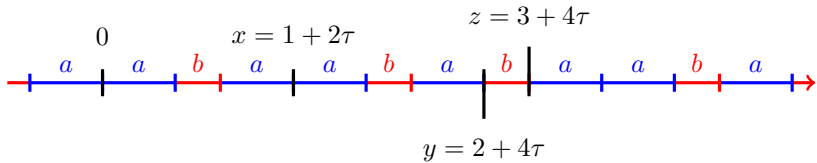
\begin{tikzpicture}[scale=0.6]
\draw[red, very thick] (-10.5,0) -- (-10,0);
\draw[blue, very thick] (-10,0) -- (-8.39,0)node[above, pos=0.5]{$a$};
\draw[blue, very thick] (-8.39,0) -- (-6.78,0)node[above, pos=0.5]{$a$};
\draw[red, very thick] (-6.78,0) -- (-5.78,0)node[above, pos=0.5]{$b$} ;
\draw[blue, very thick] (-5.78,0) -- (-4.16,0)node[above, pos=0.5]{$a$};
\draw[blue, very thick] (-4.16,0) -- (-2.55,0)node[above, pos=0.5]{$a$};
\draw[red, very thick] (-2.55,0) -- (-1.55,0)node[above, pos=0.5]{$b$} ;
\draw[blue, very thick] (-1.55,0) -- (.06,0)node[above, pos=0.5]{$a$};
\draw[red, very thick] (.06,0) -- (1.06,0)node[above, pos=0.5]{$b$} ;
\draw[blue, very thick] (1.06,0) -- (2.67,0)node[above, pos=0.5]{$a$};
\draw[blue, very thick] (2.67,0) -- (4.28,0)node[above, pos=0.5]{$a$};
\draw[red, very thick] (4.28,0) -- (5.28,0)node[above, pos=0.5]{$b$} ;
\draw[blue, very thick] (5.28,0) -- (6.89,0)node[above, pos=0.5]{$a$};
\draw[red, ->, very thick] (6.89,0) -- (7.4,0);
\draw[blue, very thick] (-10,.2) -- (-10,-.2);
\draw[blue, very thick] (-8.39,.2) -- (-8.39,-.2);
\draw[blue, very thick] (-6.78,.2) -- (-6.78,-.2);
\draw[red, very thick] (-5.78,.2) -- (-5.78,-.2);
\draw[blue, very thick] (-4.16,.2) -- (-4.16,-.2);
\draw[blue, very thick] (-2.55,.2) -- (-2.55,-.2);
\draw[red, very thick] (-1.55,.2) -- (-1.55,-.2);
\draw[blue, very thick] (.06,.2) -- (.06,-.2);
\draw[red, very thick] (1.06,.2) -- (1.06,-.2);
\draw[blue, very thick] (2.67,.2) -- (2.67,-.2);
\draw[blue, very thick] (4.28,.2) -- (4.28,-.2);
\draw[red, very thick] (5.28,.2) -- (5.28,-.2);
\draw[blue, very thick] (6.89,.2) -- (6.89,-.2);
\draw[black, very thick] (-8.39,.3) -- (-8.39,-.3)node[above, pos=-0.5]{$0$};
\draw[black, very thick] (-4.16,.3) -- (-4.16,-.3)node[above, pos=-0.5]{$x = 1+2\tau$};
\draw[black, very thick] (.06,.3) -- (.06,-.8)node[below, pos=1.2]{$y = 2+4\tau$};
\draw[black, very thick] (1.06,.8) -- (1.06,-.3)node[above, pos=-.2]{$z = 3+4\tau$};
\end{tikzpicture}

\end{example}

\vspace{.1in}

\begin{example}
A larger example solution to Pythagoras' equation in ${\ZZ}[\tau]$ is as follows:

$$(2+12\tau)^2 + (11+8\tau)^2 = (3+18\tau)^2$$

As in the previous example, multiply $\sigma(x) = 2+12\tau', \sigma(y) = 11+8\tau', \sigma(z) = 3+18\tau'$ by $(\tau')^n = \big(\frac{1 - \sqrt{5}}{2}\big)^n$ for some $n \in {\ZZ}$. After testing, we have found that for $n = 6$ the results fall into the window, as shown below. 

\begin{align*}
    (2+12\tau') \cdot (\tau')^6 &\simeq -0.302 \\
    (11+8\tau') \cdot (\tau')^6 &\simeq 0.338 \\
    (3+18\tau') \cdot (\tau')^6 &\simeq -0.458.
\end{align*}

\vspace{.1in}
Then we have

\begin{align*}
     \big((2+12\tau')\cdot(\tau')^6\big)^2 + \big((11+8\tau')\cdot(\tau')^6\big)^2 &= \big((3+18\tau')\cdot(\tau')^6\big)^2 \\
     \implies (106+172\tau')^2 + (119+192\tau')^2 &= (159+258\tau')^2
\end{align*}

and $(106+172\tau, 119+192\tau, 159+258\tau)$ is another example of a triple in the Fibonacci model set satisfying Pythagoras' equation. 

\end{example}

\vspace{.1in}
\subsection{Fermat's Last Theorem}
\hfill\\
\vspace{0.01cm}

In ${\ZZ}[\tau]$ and the Fibonacci model set Fermat's Last Theorem is not true. In particular, an example for the third power Diophantine equation in ${\ZZ}[\tau]$ is shown below.

\begin{example}
Let $x = 4+3\tau$, $y = 5+6\tau$, and $z = 6+6\tau$. Then,
\begin{align*}
    (4+3\tau)^3 + (5+6\tau)^3 &= 64 + 144\tau + 108\tau^2 + 27\tau^3 + 125 + 450\tau + 540\tau^2 + 216\tau^3 \\
    &= 189 + 594\tau + 648(\tau+1) + 243(2\tau+1) \\ 
    &= 1080 + 1728\tau = 216 + 648\tau + 648\tau^2 + 216\tau^3 \\
    &= (6+6\tau)^3
\end{align*}

As in the previous examples, multiply $\sigma(x) = 4+3\tau', \sigma(y) = 5+6\tau', \sigma(z) = 6+6\tau'$ by $(\tau')^n = \big(\frac{1 - \sqrt{5}}{2}\big)^n$ for some $n \in {\ZZ}$. After testing $n = 1, 2, 3$ we have found that for $n = 3$ the results fall into the window, as shown below.

\begin{align*}
    (4+3\tau')\cdot(\tau')^3 &\simeq -0.507 \\
    (5+6\tau')\cdot(\tau')^3 &\simeq -0.305 \\
    (6+6\tau')\cdot(\tau')^3 &\simeq -0.541.
\end{align*}

Then we have

\begin{align*}
     \big((4+3\tau')\cdot(\tau')^3\big)^3 + \big((5+6\tau')\cdot(\tau')^3\big)^3 &= \big((6+6\tau')\cdot(\tau')^3\big)^3 \\
     \implies (10+17\tau')^3 + (17+28\tau')^3 &= (18+30\tau')^3
\end{align*}

and $(10+17\tau, 17+28\tau, 18+30\tau)$ is an example of a triple in the Fibonacci model set which is a counterexample to Fermat's Last Theorem for the third power. Moreover, $\sigma (4+3\tau) \cdot (\tau')^2=10\tau'+7 \notin [-1, \tau-1)$, and so by Theorem \ref{prop:4.8}, $((4+3\tau)\tau^n)^2+((5+6\tau)\tau^n)^2=((6+6\tau)\tau^n)^2$ is in the Fibonacci model set if and only if $n \geq 3$.
\end{example}
\begin{prop}
 There exist infinitely many nontrivial solutions to $x^3+y^3=z^3$ in the Fibonacci model set. 
\end{prop}
 \medskip

The authors used computer algorithms to check for counterexamples to Fermat's Last Theorem in ${\ZZ}[\tau]$ in the fourth and fifth powers. Each program tested all integer coefficients of $x, y, z \in {\ZZ}[\tau]$ with the bounds $[-100, 100]$, and no solutions were produced for either power. 

\section*{Acknowledgements}

This work was supported by the Level UP program with funding from the Government of Canada, and the authors are grateful for the support. The authors would also like to extend their gratitude to Chris Ramsey and Nicolae Strungaru for their supervision.

\end{document}